\documentclass{amsart}
\usepackage{verbatim,amscd,amsmath,amsthm,amstext,amssymb,amsfonts,
}
\usepackage{color}
\usepackage{booktabs}
\usepackage{multirow}

\newtheorem{theorem}{Theorem}[section]
\newtheorem{corollary}[theorem]{Corollary}
\newtheorem{lemma}[theorem]{Lemma}
\newtheorem{proposition}[theorem]{Proposition}

\theoremstyle{definition}
\newtheorem{definition}[theorem]{Definition}

\newtheorem{pf}{Proof}

\newtheorem{remark}{Remark}
\newtheorem*{acknow}{Acknowledgments}

\numberwithin{equation}{section}

\begin{document}
\title[Lagrangian Bonnet pairs in complex space forms] {Lagrangian Bonnet pairs in
complex space forms}
\author{Huixia He}
\address{School of Mathematics and Systems Science of Beijing University of Aeronautics and Astronautics ,
LIMB of the Ministry of Education, Beijing 100083, P.R.~China}
\email{hehx@buaa.edu.cn}
\author{Hui Ma}
\address{Department of Mathematical Sciences,
Tsinghua University, Beijing, 100084, P.R.~China}
\email{hma@math.tsinghua.edu.cn}
\author{Erxiao Wang}
\address{Department of Mathematics, Hong Kong Univ. of Sci. and
Tech., Clear Water Bay, Kowloon, Hong Kong}
\email{maexwang@ust.hk}
\thanks{*Corresponding author: Huixia He}
\keywords{Lagrangian submanifolds, complex space forms, Bonnet pairs}
 \subjclass[2010]{Primary
53C40, Secondary 53C42, 53D12}

\begin{abstract}
 In this paper we first give a Bonnet theorem for conformal Lagrangian surfaces in complex space forms, then we show that any compact Lagrangian
 surface in the complex space form admits at most one other global isometric Lagrangian surface with the same mean curvature form, unless
the Maslov form is conformal. These two Lagrangian surfaces are then
called Lagrangian Bonnet pairs. We also studied the question about
Lagrangian Bonnet surfaces in $\tilde{M}^2(4c)$, and obtain  some
interesting results.
\end{abstract}

\maketitle

\section{Introduction}
 In surface theory of $3$-dimensional Euclidean space $\mathbb R^3$,
 the series work by Bonnet,  Cartan and Chern (\cite{Bonnet, Cartan, Chern}) show
 that mean curvature and metric is sufficient to determine an  oriented surface generically,
 except three cases:  \textit{constant mean curvature surfaces},
\textit{Bonnet surfaces}, which admits a nontrivial isometric
deformation preserving the mean curvature function, and
\textit{Bonnet pairs}, which are exactly two non-congruent isometric
surfaces with the same mean curvature function. CMC surfaces have
been investigated intensively by various methods; Bonnet surfaces
have been treated using techniques of integrable systems theory and
extended to $S^3$ and $H^3$, and recently have been generalized to
the homogeneous $3$-manifold with a $4$-dimensional isometry group
have been investigated intensively by various methods
(\cite{Cartan}, \cite{Chern}, \cite{Chenli}, \cite{ChenPeng},
\cite{GMM}). But much less is known about Bonnet pairs. Bobenko
\cite{Bobenko} takes some steps towards attacking this problem by
treating Bonnet pairs as integrable systems. However,  it is still
an open question whether compact Bonnet pairs exist in $\mathbb
R^3$. Since the theory of Bonnet pairs in $\mathbb R^3$ is closely
related to isothermic surfaces in $S^3$ ,  it was generalized to
$S^3$ in \cite{KPP}. Lawson and Tribuzy (\cite{LT}) showed that any
compact oriented surface in the $3$-dimensional real form $M^3(c)$
with nonconstant mean curvature, admits at most two surfaces with
the given metric and mean curvature. Springborn (\cite{Springborn})
showed that helicoidal immersed tori are compact Bonnet pairs in
$3$-dimensional sphere.

By the investigation of Lagrangian surfaces in the complex
projective plane $\mathbb{C}P^2$, we find it is very interesting to
consider the analog problem for conformal Lagrangian surfaces in the
complex space form $\tilde{M}^2(4c)$, where $c=1$, $0$ or $-1$.
Explicitly, which data are sufficient to determine a conformal
Lagrangian immersion of a Riemannian surface in $\tilde{M}^2(4c)$ up
to rigidity motions. The first two authors introduced a new concept
of \emph{Lagrangian Bonnet pairs} in $\mathbb CP^2$ in a similar
spirit and derived a Lawson-Tribuzy type theorem (\cite{hm}). In
this paper, we first generalize this result to the surfaces in the
complex Euclidean plane $\mathbb{C}^2$ and the complex hyperbolic
plane $\mathbb CH^2$. On the other hand, we also study the
\textit{Lagrangian Bonnet surfaces} in $\tilde{M}^2(4c)$, i.e. to
Lagrangian surfaces in complex space forms possessing one-parameter
families of isometries preserving the mean curvature form, and
obtain several interesting results.

Remark that Lagrangian submanifolds in complex space forms with
conformal Maslov form were deeply studied in \cite{CMU},
\cite{CU93}, \cite{CU}, \cite{CU97}, \cite{RU} and etc.

\section{Lagrangian surfaces in complex space forms}

\subsection{Lagrangian surface in $\mathbb{C}^2$}
Identify $\mathbb{C}^2$ with $\mathbb{R}^4$ equipped with its standard inner product $\langle\, ,\,\rangle$,
the standard complex structure $J$ and K\"{a}hler form $\omega$.
Let $f:\Sigma \rightarrow \mathbb{C}^2$ be a Lagrangian immersion of an
oriented surface. The induced metric on $\Sigma$
generates a complex structure with respect to which the metric is
$g=2e^u dzd{\bar z}$, where $z=x+iy$ is a local complex coordinate
on $\Sigma$ and $u$ is a real-valued function defined on $\Sigma$ locally.
Let $f_z$ and $f_{\bar z}$ denote the complexified tangent vector, where
$$\frac{\partial}{\partial z}=\frac{1}{2}(\frac{\partial}{\partial x}-i\frac{\partial}{\partial y}),
\quad \quad \frac{\partial}{\partial {\bar z}}=\frac{1}{2}(\frac{\partial}{\partial x}+i\frac{\partial }{\partial y}).$$
Complexify the product $\langle \, , \, \rangle$ and the complex structure on $\mathbb{R}^4$ to $\mathbb{C}^4$, still denoted by
$\langle\, , \rangle$ and $J$.
Then the metric $g$ is conformal that gives
\begin{equation}\label{conformal}
\langle f_{z},   f_ z \rangle=\langle f_{\bar z}, f_{\bar z}\rangle =0,
\quad \quad \langle f_z, f_{\bar z}\rangle=e^u.
\end{equation}
Moreover, a conformal immersion $f: \Sigma\rightarrow \mathbb{C}^2$ is Lagrangian if and only if
\begin{equation}\label{Lagrangian}
\langle f_z, Jf_{\bar z} \rangle=0.
\end{equation}

Thus the vectors $f_z, f_{\bar z}, Jf_z$ as well as $Jf_{\bar{z}}$
define an
orthogonal moving frame on the surface.
Denote $\sigma=(f_z, f_{\bar{z}}, Jf_z, Jf_{\bar{z}})$, which due to
\eqref{conformal} and \eqref{Lagrangian} satisfies the following
equations\begin{equation}\label{eq:frame}\sigma_{z}=\sigma {\mathcal
U}, \quad \sigma_{\bar z}=\sigma{\mathcal V} ,
\end{equation}
\begin{equation}\label{eq:UV}
{\mathcal U}=\left(\begin{array}{cccc}
            u_z&0& -\phi & -\bar{\phi} \\
            0& 0  & -e^{-u}\psi& -\phi\\
            \phi & \bar{\phi} &u_z &0\\
            e^{-u}\psi&\phi&0&0
            \end{array}
            \right),
 {\mathcal V}=\left(\begin{array}{cccc}
     0&0 & -\bar{\phi} & -e^{-u}\bar{\psi}\\
    0 & u_{\bar{z}} & - \phi&-\bar{\phi}\\
   \bar{\phi} & e^{-u}\bar{\psi} & 0&0\\
     \phi&\bar{\phi}&0&u_{\bar{z}}
    \end{array}\right),
\end{equation}
where
\begin{equation*}
 \phi=e^{-u}\langle f_{z\bar z},  Jf_{z} \rangle,
\quad  \psi=\langle f_{zz}, Jf_z\rangle.
\end{equation*}

The one-form $\Phi=\phi dz$ and the cubic differential
$\Psi=\psi dz^3$ are globally defined on $\Sigma$ which are independent
of the choice of the complex coordinates and $U(2)$-invariant. We call $\Phi$ and $\Psi$ mean
curvature form and Hopf differential of $f$, respectively.

By straightforward calculations we get the following integrability conditions:
\begin{eqnarray*}
\phi_{\bar z}-{\bar\phi}_{z}&=&0,\label{eq:integ1_C2}\\
u_{z\bar z}+|\phi|^2-e^{-2u}|\psi|^2 &=&0,\label{eq:integ2_C2}\\
 e^{-2u}\psi_{\bar z}&=&(e^{-u}\phi)_z.\label{eq:integ3_C2}
\end{eqnarray*}

\subsection{Lagrangian surface in $\mathbb{C}H^2$}

Let $(\mathbb CH^2, h, J, \omega)$ be the $2$-dimensional complex hyperbolic space endowed with the Fubini-Study metric $h$ of constant
holomorphic sectional curvature $-4$, where $J$ denotes the standard complex structure of $\mathbb CH^2$ and $\omega$ is the K\"{a}hler form
given by $\omega(X,Y)=h(JX,Y)$ for any tangent vectors $X$ and $Y$.

Let
$$\mathbb{H}^5_1=\{Z\in \mathbb C^3 \, \vert \, (Z,Z)=-1\},$$
where $( \, , \, )$ denotes the Hermitian inner product defined by
$$(Z, W)=-z_0\bar{w}_0+z_1\bar{w}_1+z_2\bar{w}_2,$$
for any $Z=(z_0, z_1,z_2), W=(w_0,w_1,w_2)\in \mathbb C^3$.
Thus $\langle Z, W\rangle=\mathrm{Re}(z, w)$ induces a metric on $\mathbb{H}^5_1$ of constant curvature $-1$.
Hence the Fubini-Study metric $h$ on $\mathbb CH^2$  is obtained from the fact that the Hopf fibration  $\pi: \mathbb{H}^5_1(-1)\to \mathbb CH^2$ is a Riemannian submersion.

Now let $f: \Sigma\rightarrow \mathbb{C}H^2$ be a Lagrangian immersion of an oriented surface.
Still denote the induced metric $g=f^* h=2e^u dzd\bar{z}$ on $\Sigma$ by $\langle \, , \, \rangle$ without causing ambiguity.
And there always exists a horizontal local lift $F: {\rm U}\rightarrow \mathbb{H}_1^5$ such that
\begin{equation}\label{horizontal}
\langle F_z , JF \rangle=\langle F_{\bar z}, JF \rangle =0,
\end{equation}
where $\rm U$ is an open set of $\Sigma$. In fact, generally, it
follows from $\Sigma$ is Lagrangian that $\langle dF, JF \rangle$ is
a closed one-form for any local lift $F$. So there exists a real function $\eta\in
C^\infty ({\rm U})$ locally such that $d\eta= \langle dF,   JF\rangle$. Then
$\tilde{F}=e^{i\eta} F$ is a horizontal local lift for $f$ to $\mathbb{H}^5_1$.

The metric $g$ is conformal that gives
\begin{eqnarray}
& &\langle F_{z} ,  F_{\bar z}\rangle  =e^u,\label{conformal1}\\
& &\langle F_z,  F_z \rangle = \langle F_{\bar z} ,
 F_{\bar z} \rangle=0 \label{conformal2}.
\end{eqnarray}
$F$ is a Lagrangian immersion means that
$$\langle F_z, JF_{\bar{z}}\rangle=\langle F_{\bar{z}}, JF_z\rangle=0.$$
Thus the vectors $F, JF, F_z, JF_z,  F_{\bar z}$ as well as $JF_{\bar z}$
define a  moving frame on the surface.  One obtains the following
frame equations:
\begin{equation}\label{eq:frame_CH2}
\begin{array}{l}
F_{zz}=u_zF_z+\phi JF_z+e^{-u}\psi
JF_{\bar{z}},\\
F_{z\bar{z}}=\bar{\phi} JF_z+\phi
JF_{\bar{z}}+e^uF,\\
F_{\bar{z}\bar{z}}=u_{\bar{z}}F_{\bar{z}}+e^{-u}\bar{\psi}
JF_z+\bar{\phi}JF_{\bar{z}},\end{array}
\end{equation}
where
\begin{equation}\label{eq:phipsi}
\quad \phi=e^{-u}\langle  F_{z\bar z},  JF_z  \rangle, \quad
\psi=\langle F_{zz},  JF_z \rangle.
\end{equation}
It is easy to see that the one-form $\Phi=\phi dz$ and the cubic differential $\Psi=\psi dz^3$ are globally defined on $\Sigma$
and  $U(1, 2)$-invariant.
We also call $\Phi$ and $\Psi$ mean curvature form and Hopf differential of $f$, respectively.

The compatibility condition of equations \eqref{eq:frame_CH2} has the following form:
\begin{eqnarray*}
\phi_{\bar z}-{\bar\phi}_z&=&0,\label{eq:integ1_CH2}\\
u_{z\bar z}+|\phi|^2-e^u-e^{-2u}|\psi|^2&=&0,\label{eq:integ2_CH2}\\
 e^{-2u}\psi_{\bar z}&=&(e^{-u}\phi)_z. \label{eq:integ3_CH2}
\end{eqnarray*}

Recall that for a conformal Lagrangian immersed surface $f:\Sigma\rightarrow \mathbb{C}P^2(4)$,
let $F$ be the horizontal local lift to $S^5(1)$ (\cite{hm}).
Still denote  by $\langle\,,\,\rangle$ the standard inner product on $\mathbb{C}^3$ and $J$ the standard complex structure.
Define the mean curvature form $\Phi=\phi dz$ and the cubic Hopf cubic form $\Psi=\psi dz^3$ by \eqref{eq:phipsi}
in terms of the corresponding $F$ and $\langle\, ,\,\rangle$.
Then we have
\begin{theorem} 
Let $f:\Sigma\rightarrow \tilde{M}^2(4c)$ be a conformal Lagrangian surface in $\tilde{M}^2(4c)$ with the metric
$g=2e^u dz d\bar{z}$ for $c=0$, $1$ or $-1$.
Then the metric $g$, mean curvature form $\Phi=\phi dz$ and cubic Hopf differential $\Psi=\psi dz^3$ satisfy the following equations:
\begin{eqnarray}
\phi_{\bar z}-{\bar\phi}_z&=&0,\label{eq:integ1_csf}\\
u_{z\bar z}+|\phi|^2+c e^u-e^{-2u}|\psi|^2&=&0,\label{eq:integ2_csf}\\
e^{-2u}\psi_{\bar z}&=&(e^{-u}\phi)_z. \label{eq:integ3_csf}
\end{eqnarray}

Conversely, given a metric $g=2e^udzd\bar{z}$, a one-form $\Phi=\phi dz$ and a cubic form $\Psi=\psi dz^3$ on $\Sigma$
satisfying \eqref{eq:integ1_csf}, \eqref{eq:integ2_csf} and \eqref{eq:integ3_csf},
there exists a Lagrangian immersion $f:\tilde{\Sigma}\rightarrow \tilde{M}^2(4c)$ from the universal cover $\tilde{\Sigma}$ of $\Sigma$
to  $\tilde{M}^2(4c)$ with the metric $g$ and mean curvature form $\Phi$ and cubic Hopf differential $\Psi$.
The immersion $f$ is unique up to isometries in  $\tilde{M}^2(4c)$.
\end{theorem}

\begin{remark}
\begin{enumerate}
\item Instead of the $1$-form $\Phi=\phi dz$, Oh \cite{O1}, Schoen and Wolfson \cite{SW} used the
famous Maslov form $\sigma_H$ defined by $\sigma_H=\omega(H,\cdot)$ and in fact $\sigma_H=-(\Phi+\bar\Phi)$, where $H$ is the mean
curvature vector field, or Castro and Urbano introduced a global
vector field ${\mathcal K}=e^{-u}\bar\phi\frac{\partial}{\partial
z}$ over $\Sigma$.
Remark that $|\phi|^2=\frac{1}{2}|H|^2 e^u$.
\item It is known (\cite{Da}) that the Maslov form of a Lagrangian submanifold in an Einstein-K\"{a}hler manifold is always closed.
This is equivalent to the first integrability equation \eqref{eq:integ1_csf} in our cases.
\item Actually the cubic differential $\Psi$ has been introduced by several authors in the study of minimal surfaces in K\"{a}hler
manifolds, for instance, Eells and Wood \cite{EW}, Chern and Wolfson \cite{CW}, etc. It corresponds to the symmetric $3$-tensor field $S(X,Y,Z):=h(II(X,Y),JZ)$ for a Lagrangian submanifold in a K\"{a}hler manifold, where $II$ is the second fundamental form of the submanifold and $h$ is the K\"{a}hler metric tensor of the ambient manifold (refer to, e.g., \cite{CU93}).
\item In order to be consistent with the notations in most publications on Lagrangian surfaces, such as \cite{ CU, CU93}, we define $\Phi$ and $\Psi$ here in terms of the complexified standard inner products. In fact, this definition is same as the one we used in \cite{MM, hm}  given by the Hermitian inner product,  only up to a factor $\sqrt{-1}$.
 \end{enumerate}
\end{remark}

Combining Bonnet theorems, we can summarize the geometric
interpretations of $\Phi$ and $\Psi$ (refer to \cite{CU93, CU, CU97,
hm, HR, Ma, O1 }):

\begin{proposition}
Let $f: \Sigma\rightarrow \tilde{M}(4c)$ be an oriented  immersed
surface $\Sigma$. Then $f$ is minimal if and only if $\Phi\equiv 0$.
 In addition, if $f$ is Lagrangian, then the following statements are equivalent:
 \begin{enumerate}
\item $f$ is Hamiltonian stationary.
\item The Maslov form $\sigma_H$ is a harmonic $1$-form.
\item The Lagrangian angle $\beta$ is a local harmonic function.
\item The mean curvature form $\Phi$ is holomorphic.
\end{enumerate}
\end{proposition}

\begin{proposition} \label{prop:psihol}
Let $f:\Sigma^2 \rightarrow \tilde{M}^2(4c)$ be a conformal Lagrangian immersion with the induced metric $g=2e^{u}dzd\bar{z}$.
Then the following statements are equivalent:
\begin{enumerate}
\item The Maslov form $\sigma_H$ is conformal.
\item The vector field $JH$ is a conformal vector field.
\item The cubic Hopf differential $\Psi$ is holomorphic.
\item $(e^{-u}\phi)_z=0$.
\end{enumerate}
\end{proposition}

\section{Lagrangian Bonnet surfaces in complex space forms}

We now suppose that we are given three isometric noncongruent
Lagrangian immersions $f_k: \Sigma\rightarrow \tilde{M}(4c)$,
$k=1,2,3$ with coinciding mean curvature one-form $\Phi$. As
conformal immersions of the same Riemann surface, they are described
by the corresponding cubic Hopf differentials $\Psi_1, \Psi_2, \Psi_3$,
the conformal metric $2e^udzd\bar z$ and the mean curvature form
$\Phi$. Since the surfaces are non-congruent the cubic Hopf differentials
differ.

It follows from \eqref{eq:integ2_csf} and \eqref{eq:integ3_csf} that

\begin{proposition}\label{prop:difference}
Each of the differences $\Psi_{ij}\equiv \Psi_i-\Psi_j$ for $1\leq
i, j \leq 3$, is a holomorphic cubic differential form on $\Sigma$.
Moreover,
\begin{equation}\label{eq:length}
|\Psi_i|=|\Psi_j|, \quad 1\leq i,j\leq 3.
\end{equation}
\end{proposition}

Due to the second statement of Proposition \ref{prop:difference} the zeros of $\Psi_k$ for $k=1,2,3$ coincide.
Clearly, these  points  are contained in the zeros of the holomorphic differential $\Psi_{ij}$.
We call this isolated points set as the umbilic points set of $f_k, k=1,2,3$.
Denote it by
$${\mathcal  U}=\{P\in \Sigma| \Psi_k(P)=0\}.$$

\bigskip

Let $\Sigma$ be a Lagrangian surface in $\tilde{M}$ with the data $\{u, \Phi,\Psi\}$ and $\Sigma^*$ be an isometric deformation of $\Sigma$
preserving the mean curvature form $\Phi$ with the data $\{u, \Phi, \Psi^*\}$.

The following result is the Lagrangian version of Tribuzy's result
(see \cite{Chenli} or \cite{hm}, \cite{T} ).

\begin{lemma}\label{lem:laplace}
Let $\Sigma$ be a Lagrangian Bonnet surface in $\tilde{M}( c )$, then
\begin{equation}\label{eq:laplace}
(\log \psi)_{z\bar{z}}=|(\log \psi)_{\bar z}|^2.
\end{equation}
and it is equivalent to
\begin{equation}\label{eq:keyformula}
(\frac{\psi_{\bar z}}{|\psi|^2})_z=0,
\end{equation}
\end{lemma}

\begin{proof}
From \eqref{eq:integ2_csf}, we have $\psi^{*}=e^{it} \psi$, where $t$ is a real-valued function 
determined up to a multiple of $2\pi$. Then \eqref{eq:integ3_csf} implies that
$$(e^{it}\psi-\psi)_{\bar z}=0,$$
which is equivalent to
\begin{equation}
(e^{it}-1)\psi_{\bar z}+e^{it}it_{\bar z}\psi=0.
\end{equation}
Thus
\begin{equation}
\tau\equiv dt-(1-e^{-it})i\frac{\psi_{\bar{z}}}{\psi}d\bar{z}+(1-e^{it})i\frac{\bar{\psi}_z}{\bar{\psi}}dz=0.
\end{equation}
The Pfaff system $\tau=0$ is completely integrable if and only if $d\tau\wedge \tau=0$, that is,
$$e^{it}[|\frac{\psi_{\bar z}}{\psi}|^2-(\frac{\bar{\psi}_z}{\bar{\psi}})_{\bar{z}}]+e^{-it}[|\frac{\bar{\psi}_{z}}{\bar\psi}|^2-(\frac{\psi_{\bar z}}{\psi})_{z}]
+(\frac{\psi_{\bar z}}{\psi})_z+(\frac{\bar{\psi}_{z}}{\bar{\psi}})_{\bar z} -|\frac{\psi_{\bar z}}{\psi}|^2-|\frac{{\bar\psi}_z}{\bar \psi}|^2=0.$$
Then \eqref{eq:keyformula} follows from the arbitrariness of $t$.
\end{proof}

\begin{corollary}Let $\Sigma$ be a Lagrangian Bonnet surface in $\tilde{M}( c )$, then
\begin{equation}
(\frac{(e^{-u}\phi)_z}{|\phi|^2+e^u(c-K)})_{z}=0.
\end{equation}
\end{corollary}

Writing $\psi=|\psi|e^{i\alpha}$. Then from \eqref{eq:laplace}, \eqref{eq:integ2_csf} and \eqref{eq:integ3_csf} we get
\begin{corollary}Let $\Sigma$ be a Lagrangian Bonnet surface in $\tilde{M}( c )$, then
\begin{eqnarray}
&&(\log (e^{3u}(e^{-u}|\phi|^2+c-K)))_{z\bar z}-\frac{2e^u|(e^{-u}\phi)_z|^2}{e^{-u}|\phi|^2+c-K}=0, \label{eq:log}\\
&&\alpha_{z\bar z}=0.
\end{eqnarray}
\end{corollary}

\begin{definition}
A Lagrangian surface $\Sigma$ in $\tilde{M}$ is isothermic if and
only if there exists locally a conformal parameter $z$ such that
$\psi dz^3$ satisfies $\psi(z,\bar{z})\in
\mathbb{R}$.\end{definition}




\begin{proposition}
A Lagrangian surface $\Sigma$ in $\tilde{M}$ is isothermic if and
only if $\mathrm{Im}(\log \psi)_{z\bar z}=0$.
\end{proposition}

\begin{proposition}
A Lagrangian surface $\Sigma$ in $\tilde{M}$ is  a Lagrangian Bonnet
surface if and only if
\begin{enumerate}
\item $\Sigma$ is Lagrangian isothermic,
\item \eqref{eq:log} holds.
\end{enumerate}
\end{proposition}

\begin{theorem}
A Lagrangian surface $\Sigma$ in $\tilde{M}$ is  a Lagrangian Bonnet
surface. Then
\begin{enumerate}
\item $\Sigma$ is Lagrangian isothermic,
\item $\frac{1}{\psi}$ is harmonic with respect to its isothermic coordinate, i.e. $(\frac{1}{\psi})_{z\bar z}=0$.
\end{enumerate}
\end{theorem}

Let $\Sigma$ be a Lagrangian Bonnet surface in $\tilde{M}$. We can
choose isothermic coordinate $z$ such that $1/\psi$ is a real-valued
harmonic function of $z$. From $(\frac{1}{\psi})_{z\bar z}=0$, we
have
\begin{equation}\label{eq:psi_inverseHarm}
\frac{1}{\psi}=h+\bar{h}
\end{equation}
 for some holomorphic function $h$.

Substituting \eqref{eq:psi_inverseHarm} into \eqref{eq:integ3_csf}, we get
$$h_z(e^{-u}\phi)_z=\bar{h}_{\bar z}(e^{-u}\bar{\phi})_{\bar z}.$$
The Codazzi equation \eqref{eq:integ3_csf} implies
\begin{equation}\label{e2u}e^{2u}=\frac{\psi_{\bar
z}}{(e^{-u}\phi)_z}=-\frac{\bar{h}_{\bar{z}}}{(h+\bar{h})^2
(e^{-u}\phi)_z}=-\frac{h_z}{(h+\bar{h})^2
(e^{-u}\bar{\phi})_{\bar{z}}}.\end{equation}
 Differentiating it twice, we have
\begin{equation}2u_z=\frac{h_{zz}}{h_z}-\frac{2h_z}{h+\bar{h}}-\frac{(e^{-u}\bar{\phi})_{\bar{z}z}}{(e^{-u}\bar{\phi})_{\bar{z}}}\end{equation}
 \begin{eqnarray} \label{u_zz}2u_{z\bar{z}}&=& \frac{h_{zz\bar{z}}}{h_z}-\frac{h_{zz}h_{z\bar{z}}}{(h_z)^2}
-\frac{2h_{z\bar{z}}}{h+\bar{h}}+\frac{2h_z\bar{h}_{\bar{z}}}{(h+\bar{h})^2}-
\left(\frac{(e^{-u}\phi)_{\bar{z}z}}{(e^{-u}\bar{\phi})_{\bar{z}}}\right)_{\bar{z}}\\
&=&\frac{2|h_z|^2}{(h+\bar{h})^2}-\left(\frac{(e^{-u}\phi)_{\bar{z}z}}{(e^{-u}\bar{\phi})_{\bar{z}}}\right)_{\bar{z}}\end{eqnarray}
From Gauss equation we know that
$$\begin{array}{rl}u_{z\bar{z}}=&e^{-2u}|\psi|^2-(|\phi|^2+ce^u)=e^{-2u}|\psi|^2-e^{2u}(|e^{-u}\phi|^2)-ce^u\\
=&-\frac{(e^{-u}\bar{\phi})_{\bar{z}}}{h_z}+\frac{h_z}{(h+\bar{h})^2}\frac{|e^{-u}\phi|^2}{(e^{-u}\bar{\phi})_{\bar{z}}}-ce^u\end{array}$$
Substituting \eqref{e2u} and \eqref{u_zz} into Gauss equation, we
get
\begin{eqnarray}\label{P1}\left(\frac{(e^{-u}\phi)_{\bar{z}z}}{(e^{-u}\bar{\phi})_{\bar{z}}}\right)_{\bar{z}}-\frac{(e^{-u}\bar{\phi})_{\bar{z}}}{h_z}
&=&\frac{2|h_z|^2}{(h+\bar{h})^2}-\frac{2h_z|e^{-u}\phi|^2}{(h+\bar{h})^2(e^{-u}\bar{\phi})_{\bar{z}}}
+2ce^u\end{eqnarray}

If $\phi$ is a real-value function on $\Sigma$,  we have the
following result.

\begin{theorem} Let $\Sigma$ be a Lagrangian Bonnet
surface in $\mathbb{C}^2$ with isothermic coordinates $z, \bar{z}$.
Then
$$w=w(z)=\int \frac{1}{h_z(z)}dz$$
is also a conformal coordinate, and the mean curvature form
$e^{-u}\phi$, 
$\frac{e^{2u}}{h_z}$ and $ \frac{|\psi|}{|h_z|}$are functions of
$$t=w+\bar{w}$$ only.\end{theorem}

\begin{pf}By the chain rule we get that
$(e^{-u}\phi)_w=h_z(e^{-u}\phi)_z$ and
$(e^{-u}\phi)_{\bar{w}}=\bar{h}_{\bar{z}}(e^{-u}\phi)_{\bar{z}}$
which implies that
$$(e^{-u}\phi)_w=(e^{-u}\bar{\phi})_{\bar w}.$$Put $2w=t+is$, then we get
$$(e^{-u}\phi)_{t}=(e^{-u}\phi)_{w}=(e^{-u}\phi)_{\bar{w}}$$which
shows that $(e^{-u}\phi)$ depends on one variable $t$ only.
$$e^{2u}=\frac{\psi_{\bar
z}}{(e^{-u}\phi)_z}=-\frac{|h_z|^2}{(h+\bar{h})^2 (e^{-u}\phi)_w}.$$
$$(e^{-u}\phi)_{z\bar{z}}=(e^{-u}\phi)_ww_z)_{\bar{z}}=(e^{-u}\phi)_{w\bar{w}}w_z\bar{w}_{\bar{z}}=\frac{(e^{-u}\phi)_{w\bar{w}}}{|h_z|^2}$$
Therefore, we have $$
\left(\frac{(e^{-u}\phi)_{\bar{z}z}}{(e^{-u}\bar{\phi})_{\bar{z}}}\right)_{\bar{z}}
=\frac{1}{|h_z|^2}\left(\frac{(e^{-u}\phi)_t^{''}}{(e^{-u}\phi)_t^{'}}\right)_t^{'}$$
Then \eqref{P1} can be rewritten as
$$\frac{1}{|h_z|^2}\left[\left(\frac{(e^{-u}\phi)_t^{''}}{(e^{-u}\phi)_t^{'}}\right)_t^{'}-(e^{-u}\phi)_t^{'}\right]=
\frac{2|h_z|^2}{(h+\bar{h})^2}-\frac{2|h_z|^2|e^{-u}\phi|^2}{(h+\bar{h})^2(e^{-u}\bar{\phi})_t^{'}}
+2c\frac{|h_z|}{(h+\bar{h})\sqrt{-(e^{-u}\phi)_t^{'}}}$$
$$\Rightarrow\quad
\left(\frac{(e^{-u}\phi)_t^{''}}{(e^{-u}\phi)_t^{'}}\right)_t^{'}-2(e^{-u}\phi)_t^{'}=
\frac{2|h_z|^4}{(h+\bar{h})^2}\left(1-\frac{|e^{-u}\phi|^2}{(e^{-u}\bar{\phi})_t^{'}}\right)
+2c\frac{|h_z|^3}{(h+\bar{h})\sqrt{-(e^{-u}\phi)_t^{'}}}$$ When
$c=0$, we can get
\begin{equation}\label{P2}\left(\frac{(e^{-u}\phi)_t^{''}}{(e^{-u}\phi)_t^{'}})_t-2(e^{-u}\phi)\right)^{'}_t=\frac{2|h_z|^4}{(h+\bar{h})^2}
\left(1-\frac{(e^{-u}\phi)^2}{(e^{-u}\phi)_t^{'}}\right)=2Q^2\left(1-\frac{(e^{-u}\phi)^2}{(e^{-u}\phi)_t^{'}}\right)\end{equation}
where $Q=\frac{|h_z|^2}{(h+\bar{h}) }$. This formula implies that
$Q$ is a function of $t$ unless
$\frac{(e^{-u}\phi)_t^{''}}{(e^{-u}\phi)_t^{'}})_t-2(e^{-u}\phi)$
and $1-\frac{(e^{-u}\phi)^2}{(e^{-u}\phi)_t^{'}}$ must vanish
identically. $(e^{-u}\phi)_t^{'}=(e^{-u}\phi)^2$ means that
$(e^{-u}\phi)_t^{'}>0$, but this is a contradiction to \eqref{e2u}.

Since $\tilde{\psi}(z, \bar{z})dz^3=\psi(w, \bar{w})dw^3$, where
$\tilde{\psi}$ is the Hopf differential form with respect
 to the coordinates $z, \bar{z}$, we get
 $$\psi(w, \bar{w})=\frac{h_z^3(w^{-1}(w))}{h(w^{-1}(w))+\bar{h}(\bar{w}^{-1}(\bar{w}))}.$$
$$\psi_{\bar{w}}=\frac{h_z|h_z|^4}{(h+\bar{h})^2}, \bar{\psi}_w=\frac{\bar{h}_{\bar{z}}|h_z|^4}{(h+\bar{h})^2}$$
$$|\psi|^2=\psi_{\bar{w}}\bar{h}_{\bar{z}}=\bar{\psi}_wh_z=\frac{|h_z|^6}{(h+\bar{h})^2}=|h_z|^2Q^2.$$
Then $\frac{2|\psi|^2}{|h_z|^2}$ is a function of $t$ only.
$$\frac{e^{2u}}{h_z}= \frac{\psi_{\bar{w}}}{h_z(e^{-u}\phi)_w}=\frac{Q^2}{(e^{-u}\phi)_t^{'}}$$
It's depends on $t$ only ,too.
\end{pf}
\begin{theorem}The holomorphic function $h=h(z)$ satisfies the
differential equation
$$h_{zz}(h+\bar{h})-h_z^2=\bar{h}_{\bar{z}\bar{z}}(h+\bar{h})-\bar{h}^2_{\bar{z}}$$ \end{theorem}
\begin{pf}Since $(e^{-u}\phi)_t^{'}-(e^{-u}\phi)^2\neq 0$, and
$Q=\frac{|h_z|^2}{(h+\bar{h})}$depends on $t$ only,
i.e.$$Q_w=Q_{\bar{w}}\quad \Leftrightarrow\quad
h_zQ_z=\bar{h}_{\bar{z}}Q_{\bar{z}}.$$
$$Q_z=\frac{h_{zz}\bar{h}_{\bar{z}}}{(h+\bar{h})}-|h_z|^2\frac{h_z}{(h+\bar{h})^2},
\quad
h_zQ_z=\frac{h_{zz}|h_z|^2}{(h+\bar{h})}-\frac{|h_z|^2(h_z)^2}{(h+\bar{h})^2}$$
$$Q_{\bar{z}}=\frac{h_{z}\bar{h}_{\bar{z}\bar{z}}}{(h+\bar{h})}-|h_z|^2\frac{\bar{h}_{\bar{z}}}{(h+\bar{h})^2},
\quad
\bar{h}_{\bar{z}}Q_z=\frac{\bar{h}_{\bar{z}\bar{z}}|h_z|^2}{(h+\bar{h})}-\frac{|h_z|^2(\bar{h}_{\bar{z}})^2}{(h+\bar{h})^2}$$This
is equivalent to
$$\frac{\bar{h}_{\bar{z}\bar{z}} }{(h+\bar{h})}-\frac{ (\bar{h}_{\bar{z}})^2}{(h+\bar{h})^2}
=\frac{h_{zz} }{(h+\bar{h})}-\frac{ (h_z)^2}{(h+\bar{h})^2}$$then we
finish the proof.\end{pf}

\begin{theorem}
Let $\Sigma$ be a compact oriented Lagrangian surface in
$\tilde{M}^2(4c)$. If its Maslov form is not conformal, then there
exist at most two noncongruent isometric immersions of $\Sigma$ in
$\tilde{M}^2(4c)$ with the mean curvature form $\Phi$.
\end{theorem}

\begin{pf}From now on we assume that $f_1, f_2$ and $f_3$ are mutually
noncongruent. We will only use the fact for the two Lagrangian
immersions $f_1$ and $f_2$. Considering $|\psi_1|^2=|\psi_2|^2$, we
may write
$$\psi_2=\psi_1 e^{i\theta},$$
where $\theta$ is well defined outside the zeros of
$|\psi_k|^2=e^{3u}(e^{-u}|\phi|^2+c-K)=e^{3u}(\frac{1}{2}|H|^2+c-K)$
modulo $2\pi$, where $K$ is the Gauss curvature of the induced
metric and $|H|$ is the length of the mean curvature vector.

We now consider
$$Q:=\frac{\psi_1-\psi_2}{\psi_1}=1-e^{i\theta},$$
which is well defined on $\Sigma\backslash {\mathcal U}$.
It follows from Lemma \ref{lem:laplace} and $\psi_{1}-\psi_2$ is holomorphic that
\begin{equation}\label{eq:sign}
\triangle \log Q=-\triangle \log \psi_1\leq 0,
\end{equation}
where $\triangle$ is the Laplacian operator on $\Sigma$.
From \eqref{eq:sign} and $\triangle \log Q=\triangle \log |Q|+i\triangle \arg Q,$ we know that
$$\triangle \log |Q|\leq 0, \quad \triangle \arg Q=0$$
on $\Sigma\backslash {\mathcal U}$.

We now observe that since $Q$ is not zero in the connected set
$\Sigma\backslash {\mathcal U}$, the function $\theta$ cannot be
zero modulo $2\pi$ in this set. Hence we can choose a continuous
branch $\theta:\Sigma\backslash {\mathcal U}\rightarrow (0, 2\pi)$.
Then there exists a continuous branch
$$\arg (Q(z))\in (-\frac{\pi}{2}, \frac{\pi}{2}),$$
for $z\in \Sigma \backslash {\mathcal U}$. In particular, $\arg Q$
is a bounded harmonic function on $\Sigma\backslash {\mathcal U}$,
where $\mathcal U$ is a discret points set. Therefore, by removable
singularities theorem, $\arg Q$ can extend to a smooth harmonic
function on $\Sigma$. $\Sigma$ is compact and connected, hence $\arg
Q$ is a constant. Moreover, $|Q-1|\equiv 1$, which follows that $Q$
is a constant. Consequently, $\Psi_1$ is holomorphic, then by
Proposition \ref{prop:psihol}, its Maslov form is conformal. This
completes the proof.\end{pf}

\section{Lagrangian Bonnet pairs}

Let $f_1, f_2$ be a Lagrangian Bonnet pair, i.e., two isometric noncongruent
Lagrangian surfaces with coinciding mean curvature form.  As
conformal immersions of the same Riemann surface
$$f_1: \Sigma \rightarrow \tilde{M}(4c), \quad f_2: \Sigma \rightarrow \tilde{M}(4c),$$
they are described by the corresponding cubic Hopf differentials $\Psi_1, \Psi_2$,
the conformal metric $2e^udzd\bar z$ and the mean curvature form $\Phi$.
Since the surfaces are non-congruent the cubic Hopf differentials differ $\Psi_1\neq\Psi_2$.

\begin{proposition}\label{prop:alpha}
Let $\Psi_1$ and $\Psi_2$ be the cubic Hopf differentials of a Lagrangian
Bonnet pair $f_{1,2}:\Sigma \rightarrow \tilde{M}(4c)$. Then there
exist a holomorphic cubic differential $h=\Psi_1-\Psi_2$ on $\Sigma$
and a smooth real valued function $\alpha:\Sigma\rightarrow \mathbb{R}$ such that
$$\Psi_1=\frac{1}{2} h(i\alpha+1), \quad \Psi_2=\frac{1}{2}h(i\alpha-1).$$
\end{proposition}
\begin{proof}
Define a smooth cubic differential
$$q=\Psi_1+\Psi_2.$$
\eqref{eq:length} implies
$$q\bar h + h\bar q =0.$$
Thus $\alpha=-i\frac{q}{h}$ is a real valued function defined on
$\Sigma\backslash {\mathcal U}_h$ where
$${\mathcal U}_h=\{P\in \Sigma |h(P)=0\}$$
is the zero set of $h$. At any $z_0\in {\mathcal U}_h$ the
holomorphic differential $h$ has the form
$$h(z)=(z-z_0)^k h_0(z)dz, \quad h_0(z_0)\neq 0, \quad k\in {\mathbb N}.$$
In a neighborhood of $z_0$ we have
$$\alpha=-\frac{i}{(z-z_0)^k} \frac{q(z)}{h_0(z)}$$
where $q$ is smooth and $h_0$ is holomorphic. Real-valuedness of $\alpha$ near $z_0$ implies
$$q(z)=(z-z_0)^kg_0(z)$$
with $g_0$ smooth, which implies the smoothness of $\alpha$ at
$z_0$. So $\alpha$ can be smoothly extended to the whole $\Sigma$.
\end{proof}

\begin{corollary}
Umbilic points of a Lagrangian Bonnet pair are isolated. The umbilic set coincides with the zero set of $h$, i.e.,
${\mathcal U}={\mathcal U}_h$.
\end{corollary}

The number $k$ which is defined above is called the {\it index} of the umbilic point. We call the zero divisor $D=(h)$ of $h$ the
{\it Umbilic divisor} of a Lagrangian Bonnet pair.

In exactly the same way as in the case of Bonnet pairs in $\mathbb
R^3$ and $\mathbb CP^2$ (see \cite{Bobenko}, \cite{hm}), for
compact Riemann surfaces, Propositions \ref{prop:difference},
\ref{prop:alpha} imply the following
\begin{proposition}
\begin{itemize}
\item[(1)] There are no Lagrangian Bonnet pairs of genus zero.
\item[(2)] Lagrangian Bonnet pairs of genus one have no umbilic points.
\item[(3)] If Lagrangian Bonnet pairs of genus $g\geq 1$ exist, the umbilic divisor $D$ is of degree $6g-6$ and its class is $D=3K$,
where $K$ is the canonical divisor.
\end{itemize}
\end{proposition}

\begin{acknow}
The first author was supported by NSFC grant No. 10701007 and No. 11071018.
The second author was supported by NSFC grant No.
No. 11271213.
And the third author was supported by the NSFC (Grant
Nos. 10941002, 11001262) and the Starting Fund for Distinguished Young
Scholars of Wuhan Institute of Physics and Mathematics (Grant No.
O9S6031001). The third author would like to thank the Hong Kong University of Science \& Technology for the  support during the project. 
\end{acknow}
\bibliographystyle{amsplain}

\end{document}